\documentclass[12pt,a4paper]{article}
\usepackage[latin1]{inputenc}
\usepackage{csquotes}
\usepackage{amsmath}
\usepackage{amsfonts}
\usepackage{amssymb}
\usepackage{amsthm}
\usepackage{psfrag}
\usepackage{graphicx}
\usepackage{enumitem}
\usepackage{indentfirst}
\usepackage{bbm}
\newcommand{\lel}{\left\langle}
\newcommand{\rir}{\right\rangle}
\newcommand{\diag}{\text{diag}}
\newcommand{\keywords}{{\it ~ Keywords:~}}
\newtheorem{thm}{Theorem}[section]
\newtheorem{definition}[thm]{Definition}
\newtheorem{lemma}[thm]{Lemma}
\newtheorem{prop}[thm]{Proposition}
\newtheorem{ass}[thm]{Assumption}
\makeatother
\title{Reflected Backward Stochastic Differential Equations for a Finite State Markov Chain Model and Applications to American Options}
\author{Dimbinirina Ramarimbahoaka \thanks{Department of Mathematics and Statistics, University of Calgary, 2500 University Drive NW, Calgary, AB, T2N 1N4, Canada. dimbikeli@gmail.com.} \and Zhe Yang \thanks{Department of Mathematics and Statistics, University of Calgary, 2500 University Drive NW, Calgary, AB, T2N 1N4, Canada. yangzhezhe@gmail.com.} \and Robert J. Elliott \thanks{Haskayne School of Business, University of Calgary, 2500 University Drive NW, Calgary, AB, T2N 1N4, Canada. relliott@ucalgary.ca}  \thanks{School of Mathematical Sciences, University of Adelaide, SA 5005, Australia.}}
\date{}
\begin{document}
\maketitle
\begin{abstract}
In this paper, we introduce a new kind of reflected backward stochastic differential equations (RBSDEs) driven by a martingale, in a Markov chain model, but not driven by Brownian motion, and give existence and uniqueness results for the new equations. Then
we discuss American options in a finite state Markov chain model, in the presence of a stochastic discount function (SDF) and using the theory of the new RBSDEs. We show that there exists a constrained super-hedging strategy for an American option, which is unique in our framework as the solution to an RBSDE.
\end{abstract}


\section{Introduction}	
\indent In 1997, El Karoui, Kapoudjian, Pardoux, Peng and Quenez \cite{KKPPQ} introduced reflected backward stochastic differential equations (RBSDEs) as follows:
\begin{enumerate}[label=\roman{*}), ref=(\roman{*})]
\item  $Y_t = \xi + \int_t^T f(s, Y_s, Z_s ) ds + K_T -K_t -\int_t^T  (Z_s, dB_s)$,  $~~~0 \leq t \leq T$;
\item  $Y_t \geq S_t $,  $~~~0 \leq t \leq T$;
\item  $\{K_t, t \in [0,T]\}$ is continuous and increasing, moreover, $K_0=0$ and \\
$\int_0^T (Y_s - S_s) dK_s = 0$.
\end{enumerate}
Here $B$ is Brownian motion. The solutions $\{(Y_t,Z_t,K_t),~t\in[0,T]\}$ are $\mathcal{F}_t$ progressively measurable processes and $Y$ is forced to stay above a process $S$ called an obstacle. To do so, a continuous increasing process $K$ is introduced in the dynamics.\\
\indent El Karoui, Pardoux and Quenez \cite{KPQ1} gave an application of RBSDEs, driven by a Brownian motion, to the optimal stopping time problem and American options. It has been shown that the price of an American option, as well as a superhedging strategy for the option, are solutions to RBSDEs. \\
\indent Hamad{\`{e}}ne and Ouknine \cite{ham1} extended continuous RBSDEs to RBSDEs with jumps. They investigated an RBSDE driven by a Brownian motion and an independent Poisson process. Moreover, instead of being continuous, the obstacle is just right continuous with left limits. They provided another solution of the problem, in Hamad{\`{e}}ne and Ouknine \cite{ham2}, using Snell envelope theory. A similar result has also been carried out by Essaky \cite{essaky}. Other significant results on BSDEs and RBSDEs with jumps are the works of Crepey and Matoussi \cite{crepy1} and Bouchard and Elie \cite{bouch}. Crepey and Matoussi\cite{crepy1} deal with more general dynamics.\\
\indent None of the above works have used a Markov chain to model the jumps. Moreover, diffusions can be approximated by Markov chains. See the work of Kushner \cite{kushner84}. Consequently, there is some motivation for discussing Markov chain models. van der Hoek and Elliott \cite{RE1} introduced a market model where uncertainties are modeled by a finite state Markov chain, rather than by Brownian motion or related jump diffusions. In this paper uncertainty is modeled using a Markov chain. Another tool used in van der Hoek and Elliott \cite{RE1} is the presence of a stochastic discount function (SDF) which implies no-arbitrage pricing. Kluge and Rogers \cite{rogers1}, Rogers \cite{rogers2}, Rogers and Zane \cite{rogers3} use the term \enquote{potential} for stochastic discount functions modelled by Markov processes. Rogers and Yousaf \cite{rogers4} combined Markov chain models and the potential approach to model interest rates and exchange rates. It is stated in Rogers and co-authors's work that taking the Markov process to be a finite state Markov chain gives better results. Moreover, the computation of the pricing formula in the potential approach is reduced to a finite weighted sum. In \cite{RE1}, stock prices are determined by the model, given the dividend paid by the stock, which in turn depends, at each time on the state of the Markov chain. SDFs are used to give the current price of future cashflows. Current prices of financial products such as bonds, foreign currencies, futures and European options were also derived in van der Hoek and Elliott \cite{RE1}. Later, van der Hoek and Elliott \cite{RE2} proved that the price of an American option in the Markov chain model with an SDF is a solution of a variational inequality driven by a system of ordinary differential equations.\\
\indent In the present work, we shall discuss American options in van der Hoek and Elliott's framework using an RBSDE approach. BSDEs in this framework were introduced by Cohen and Elliott \cite{Sam1} as
\[Y_t = \xi + \int_t^T f(u, Y_u,Z_u) du - \int_t^T Z'_{u-} dM_u,~~~ t \in [0,T],\]
where, $f$ is the driver, $\xi$ is the terminal condition and $M$ is a vector martingale given by the dynamics of the Markov chain. \\
\indent An, Cohen and Ji \cite{An} discuss American options using the theory of RBSDEs, for the Markov chain, in discrete time. This approach, as well as the above results on RBSDEs for Brownian motion, have not been investigated in a finite state Markov chain framework with a SDF in continuous time. Also, in the American option problem, as the holder of the option has the freedom to exercise at any time prior to maturity, most studies focus on determining the optimal exercise time for the holder and its associated optimal price. Instead of determining the option price, we consider the other party's side of the contract and show the existence of a superhedging strategy which covers the option's payoff at any time prior to maturity, in case the holder exercises the option. \\
\indent The sections of the paper are as follows: In Section 2, we present the Markov chain model and some preliminary results.  Section 3 establishes the existence and uniqueness of solutions for RBSDEs under the Markov chain model, and in section 4, we discuss an application to American options, where we show that a superhedging strategy exists as the solution to an RBSDE with the Markov chain noise.
\section{The  Model and Some Preliminary Results.}\label{prelim}
\subsection{The Markov Chain}
\indent Consider a continuous time financial market where randomness is modeled by a finite state Markov chain. Following van der Hoek and Elliott \cite{RE1, RE2}, we assume the finite state Markov chain $X=\{X_t: t\geq 0 \}$ is defined on the probability space $(\Omega,\mathcal{F},P)$ and the state space of $X$ is identified with the set $\{e_1,e_2\cdots,e_N\}$ in $\mathbb{R}^N$, where $e_i=(0,\cdots,1\cdots,0) ' $ with 1 in the $i$-th position. Then the  Markov chain has the semimartingale representation:
\begin{equation}\label{semimartingale}
X_t=X_0+\int_{0}^{t}A_uX_udu+M_t.
\end{equation}
Here, $A=\{A_t, t\geq 0 \}$ is the rate matrix of the chain $X$ and $M$ is a vector martingale (see Elliott, Aggoun and Moore \cite{RE4}). We assume the elements $A_{ij}(t)$  of $A$ are bounded. Then the martingale $M$ is square integrable.
Take $\mathcal{F}_t=\sigma\{X_u | 0\leq u \leq t\}$ to be the $\sigma$-algebra  generated by the Markov process $X=\{X_t\}$ and $\{\mathcal{F}_t\}$ to be the filtration on $(\Omega,\mathcal{F},P)$. Since $X$ is right continuous and has left limits (written RCLL), the filtration $\{\mathcal{F}_t\}$ is also right-continuous. \\
\indent We refer the reader to Buchanan and Hildebrandt \cite{Egg} for the proof of the following lemma.
\begin{lemma} \label{polya}
If a sequence $f_n(x)$ of monotonic functions converges to a continuous function $f(x)$ in $[a,b]$, then this convergence is uniform.
\end{lemma}
\indent The following is given in Elliott \cite{elliott} as Lemma 2.21:
\begin{lemma}\label{indistinguish}
Suppose $V$ and $Y$ are real valued processes defined on the same probability space $(\Omega, \mathcal{F},P)$ such that for every $t \geq 0$, $V_t = Y_t$, a.s. If both processes are right continuous, then $V$ and $Y$ are indistinguishable, that is:
$$P(V_t = Y_t,~ \text{for any}~ t\geq 0)=1. $$
\end{lemma}
Denote by $P'$ the transpose of any $\mathbb{R}^{n\times p}$ matrix $P$ for any $p,n \in \mathbb{N}$, $\text{diag} (x)$ for any $x\in \mathbb{R}^n$, the matrix whose diagonal components are the entries of the vector $x$ and the remaining components are zero and similarly $\text{diag} (M)$ for any $M \in \mathbb{R}^{n\times n}$ the square matrix whose diagonal components are those of $M$ and the remaining components are zero.\\
\indent For our Markov chain $X_t \in \{e_1,\cdots,e_N\}$, note that $X_t X'_t = \diag(X_t)$. Also, from \eqref{semimartingale}
$dX_t = A_t X_t dt + dM_t$. Then,
\begin{align}\label{1}
\nonumber X_tX'_t &= X_0X'_0 + \int_0^t X_{u-} dX'_u + \int_0^t (dX_{u}) X'_{u-} + \sum_{0 < u \leq t} \Delta X_u \Delta X'_u \\
 \nonumber &= \diag(X_0) + \int_0^t X_u (A_uX_u)' du + \int_0^t X_{u-} dM'_u \\
\nonumber& + \int_0^t A_u X_u X'_{u-} du + \int_0^t (dM_u) X'_{u-} + [X,X]_t\\
\nonumber
& = \diag (X_0) + \int_0^t X_u X'_u A'_u du + \int_0^t X_{u-} dM'_u \\
& + \int_0^t A_u X_u X'_{u-} du + \int_0^t (dM_u) X'_{u-} + [X,X]_t - \lel X,X\rir_t + \lel X,X \rir_t.
\end{align}
Here, $\lel X, X\rir$ is the unique predictable process such that $[X,X]-\lel X,X \rir$ is a martingale and write
\begin{equation}\label{L_t}
L_t = [X,X]_t - \lel X,X\rir_t, \quad t \in [0,T].
\end{equation}
 However, we also have:
\begin{equation}\label{2}
X_tX'_t = \diag (X_t) = \diag(X_0) + \int_0^t \diag (A_u X_u)  du + \int_0^t \diag(M_u).
\end{equation}
Equating the predictable terms in \eqref{1} and \eqref{2}, we have
\begin{equation}\label{3}
\lel X, X\rir_t =  \int_0^t \diag(A_uX_u) du - \int_0^t \diag({X_u}) A'_u du - \int_0^t A_u \diag(X_u) du.
\end{equation}
\indent Let $\Psi$ be the matrix
\begin{equation}\label{Psi}\Psi_t = \diag(A_tX_t)- \diag(X_t)A'_t - A_t \diag(X_t).
\end{equation}
Then $d \lel X,X \rir_t = \Psi_t dt$. For any $t>0$,  Cohen and Elliott \cite{Sam1, Sam3}, define the semi-norm $\|.\|_{X_t}$, for
$C, D \in \mathbb{R}^{N\times K}$ as :
\begin{align*}
\lel C, D\rir_{X_t} & = Tr(C' \Psi_t D), \\
\|C\|^2_{X_t} & = \lel C, C\rir_{X_t}.
\end{align*}
We only consider the case where $C \in \mathbb{R}^N$, hence we
introduce the semi-norm $\|.\|_{X_t}$ as:
\begin{align}\label{normC}
\nonumber
\lel C, D\rir_{X_t} & = C' \Psi_t D, \\[2mm]
\|C\|^2_{X_t} & = \lel C, C\rir_{X_t}.
\end{align}
It follows from equation \eqref{3} that
\[\int_t^T \|C\|^2_{X_s} ds = \int_t^T  C' d\lel X, X\rir_s C.\]
For $n \in \mathbb{N}$, denote by $|\cdot|_n$ the Euclidian norm in $\mathbb{R}^n$ and by $\|\cdot\|_{n\times n}$ the norm in $\mathbb{R}^{n \times n}$ such that $\|\Psi\|_{n\times n}= \sqrt{Tr(\Psi' \Psi)}$ for any $\Psi \in \mathbb{R}^{n \times n}$.\\
\indent The following lemma is Lemma 3.5 in \cite{zhedim}.
\begin{lemma}\label{normbound}
For any $C \in \mathbb{R}^N$,
$$ ~~~~\|C\|_{X_t} \leq \sqrt{3m} |C|_N, ~~\text{ for any }t\in[0,T],$$
where $m>0$ is the bound of $\|A_t\|_{N\times N}$, for any $t\in[0,T]$.
\end{lemma}
The proof of the following lemma is found in \cite{Sam3}:
\begin{lemma}\label{Z2}
For $Z$, a predictable process in $\mathbb{R}^N$, verifying:
 \[E \left[ \int_0^t \|Z_u\|^2_{X_u} du\right] < \infty,\]
we have:
\begin{equation*}
 E \left[\left(\int_0^t  Z'_{u} dM_u  \right)^2\right] = E \left[ \int_0^t \|Z_u\|^2_{X_u} du\right].
\end{equation*}
\end{lemma}
Denote by $\mathcal{P}$, the $\sigma$-field generated by the predictable processes defined on $(\Omega, P, \mathcal{F})$ and with respect to the filtration $\{\mathcal{F}_t\}_{t \in [0,\infty)}$. For $t\in[0,\infty)$, consider the following spaces: \\[2mm]
$ L^2(\mathcal{F}_t): =\{\xi;~\xi$ is a $ \mathbb{R} \text{-valued}~ \mathcal{F}_t $-measurable random variable such that $ E[|\xi|^2]< \infty\};$\\[2mm]
$L^2_{\mathcal{F}}(0,t;\mathbb{R}): =\{\phi:[0,t]\times\Omega\rightarrow\mathbb{R};~ \phi$ is an adapted and RCLL process with  $E[\int^t_0|\phi(s)|^2ds]<+\infty\}$;\\[2mm]
$P^2_{\mathcal{F}}(0,t;\mathbb{R}^N): =\{\phi:[0,t]\times\Omega\rightarrow\mathbb{R}^N;~ \phi $ is a predictable process with  $E[\int^t_0\|\phi(s)\|_{X_s}^2ds]<+\infty\}.$
\subsection{BSDEs for the Markov Chain Model.}\label{bsdeMC}
\indent Consider a one-dimensional BSDE with the Markov chain noise
as follows:
\begin{equation}\label{BSDEMC}
Y_t = \xi + \int_t^T f(u, Y_u, Z_u ) du -\int_t^T  Z'_{u} dM_u
,~~~~~t\in[0,T].
\end{equation}
Here the terminal condition $\xi$ and the coefficient $f$ are known. \\
\indent Lemma \ref{existence} (Theorem 6.2 in Cohen and Elliott \cite{Sam1}) gives the existence and uniqueness result of solutions for BSDEs
driven by Markov chains.
\begin{lemma}\label{existence}
Assume $\xi\in L^2(\mathcal{F}_T)$ and the predictable
function $f: \Omega \times [0, T] \times \mathbb{R} \times
\mathbb{R}^N \rightarrow \mathbb{R}$ satisfies a Lipschitz
condition, in the sense that there exists some constants $l_1, l_2>0$  such
that for each $y_1,y_2 \in \mathbb{R}$ and $z_1,z_2 \in
\mathbb{R}^{N}$, $t\in[0,T]$,
\begin{equation}\label{Lipchl}
|f(t,y_1,z_1) - f(t, y_2, z_2)| \leq l_1 |y_1-y_2| + l_2 \|z_1
-z_2\|_{X_t}.
\end{equation}
We also assume $f$ satisfies
\begin{equation}\label{finite}
 E [ \int_0^T |f^2(t,0,0)| dt] <\infty.
\end{equation} Then  there exists a solution $(Y, Z)\in L^2_{\mathcal{F}}(0,T;\mathbb{R})\times P^2_{\mathcal{F}}(0,T;\mathbb{R}^N)$
to BSDE (\ref{BSDEMC}). Moreover, this solution is
unique up to indistinguishability for $Y$ and equality $d\langle
X,X\rangle_t$ $\times\mathbb{P}$-a.s. for $Z$.
\end{lemma}
The following lemma as an extension of the above lemma to stopping times can be found in Cohen and Elliott \cite{Sam3}.
\begin{lemma}\label{BSDEST} Let $\tau >0$ be a stopping time such that there exists a real value $T$ such that $P(\tau > T)=0$.
Under the assumptioms of Lemma \ref{existence} with changing $T$ into $\tau$, BSDE for the Markov chain with stopping time
\begin{equation}
Y_t = \xi + \int_{t\wedge\tau}^{\tau} f(s, Y_s, Z_s ) ds -\int_{t\wedge\tau}^{\tau}  Z'_{s} dM_s
,~~~~~t\geq 0.
\end{equation}
has a solution $(Y, Z)\in L^2_{\mathcal{F}}(0,\tau;\mathbb{R})\times P^2_{\mathcal{F}}(0,\tau;\mathbb{R}^N)$. Moreover, this solution is
unique up to indistinguishability for $Y$ and equality $d\langle
X,X\rangle_t$ $\times\mathbb{P}$-a.s. for $Z$.
\end{lemma}
See Campbell and Meyer \cite{campbell} for the following definition:
\begin{definition}[Moore-Penrose pseudoinverse]\label{defMoore}
The Moore-Penrose pseudoinverse of a square matrix $Q$ is the matrix $Q^{\dagger}$ satisfying the properties:\\[2mm]
  1) $QQ^{\dagger}Q = Q$ \\[2mm]
  2) $Q^{\dagger}QQ^{\dagger} = Q^{\dagger}$ \\[2mm]
  3) $(QQ^{\dagger})' = QQ^{\dagger}$ \\[2mm]
  4) $(Q^{\dagger}Q)'=Q^{\dagger}Q.$
\end{definition}
\begin{ass}\label{ass0}
Assume the Lipschitz constant $l_2$ of the driver $f$ given in \eqref{Lipchl} satisfies  $$~~~~~~l_2\|\Psi_t^{\dagger}\|_{N \times N} \sqrt{6m}< 1, ~~~\text{ for any }~t \in [0,T],$$ where $\Psi$ is given in \eqref{Psi} and $m>0$ is the bound of $\|A_t\|_{N\times N}$, for any $t\in[0,T]$.
\end{ass}
\indent The following lemma, which is a comparison result for BSDEs driven by a Markov chain, is found in Yang, Ramarimbahoaka and Elliott \cite{zhedim}.\\
\begin{lemma} \label{CT} For $i=1,2,$ suppose $(Y^{(i)},Z^{(i)})$ is the solution of the
BSDE:
$$Y^{(i)}_t = \xi_i + \int_t^T f_i(s, Y^{(i)}_s, Z^{(i)}_s ) ds
- \int_t^T (Z_{s}^{(i)})' dM_s,\hskip.4cmt\in[0,T].$$
Assume $\xi_1,\xi_2\in L^2(\mathcal{F}_T)$, and $f_1,f_2:\Omega \times [0,T]\times \mathbb{R}\times \mathbb{R}^N \rightarrow \mathbb{R}$ satisfy some conditions such that the above two BSDEs have unique solutions. Moreover assume $f_1$ satisfies \eqref{Lipch} and Assumption \ref{ass0}.
If $\xi_1 \leq \xi_2 $, a.s. and $f_1(t,Y_t^{(2)}, Z_t^{(2)}) \leq f_2(t,Y_t^{(2)}, Z_t^{(2)})$, a.e., a.s., then
$$P( Y_t^{(1)}\leq Y_t^{(2)},~~\text{ for any } t \in [0,T])=1.$$
\end{lemma}
\section{RBSDEs driven by the Markov Chains}\label{section4}
\indent We now introduce an RBSDE for the Markov Chain:
\begin{enumerate}[label=\roman{*}), ref=(\roman{*})]
\item  $V_t = \xi + \int_t^T f(u, V_u, Z_u ) du + K_T -K_t -\int_t^T  Z'_u dM_u$,  $~~~0 \leq t \leq T$ ;
\item  $V_t \geq G_t $,  $0 \leq t \leq T$;
\item  $\{K_t, t \in [0,T]\}$ is continuous and increasing, moreover, $K_0=0$ and \\
$\int_0^T (V_u - G_u) dK_u = 0$.
\end{enumerate}
\indent We want to show the existence and uniqueness of the solution $(V,Z,K)$ of above equation under some conditions on $\xi,f$ and $G$ .
\begin{thm}
Suppose we have:
\begin{enumerate}
 \item $\xi\in L^2(\mathcal{F}_T)$,
\item  a $\mathcal{P} \times \mathcal{B}(\mathbb{R}^{1+N})$ measurable function $f: \Omega \times [0, T] \times \mathbb{R} \times \mathbb{R}^N \rightarrow \mathbb{R}$ which is Lipschitz continuous, with constants $c'$ and $c''$, in the sense that, for any $t \in [0,T]$, $v_1,v_2 \in \mathbb{R}$ and $z_1,z_2 \in \mathbb{R}^N$, $t\in[0,T]$,\emph{}
\begin{equation}\label{Lipch}
|f(t,v_1,z_1) - f(t, v_2, z_2)| \leq c'|v_1-v_2| + c''\|z_1 -z_2\|_{X_t}
\end{equation}
and $c''$ satisfies
 \begin{equation}\label{c''}c''\|\Psi_t^{\dagger}\|_{N \times N} \sqrt{6m}< 1, ~~~\text{ for any }~t \in [0,T],\end{equation}
 where $\Psi$ is given in \eqref{Psi} and $m>0$ is the bound of $\|A_t\|_{N\times N}$, for any $t\in[0,T]$.
\item \begin{equation}\label{Con_f}
 E \left[ \int_0^T |f^2(t,0,0)| dt\right] < \infty,
\end{equation}
\item a process $G$ called an \enquote{obstacle} which satisfies
\begin{equation}\label{Con_g}
 E \left[ \sup_{0\leq t \leq T} (G_t^+)^2 \right] < \infty.
\end{equation}
\end{enumerate}
Then there exists a solution $(V,Z,K)$, $V$ adapted and RCLL and $Z$ predictable, of the RBSDE i), ii), iii) above such that $V \in L^2_{\mathcal{F}}(0,T;\mathbb{R})$, $K_T\in L^2(\mathcal{F}_T)$ and $Z \in P^2_{\mathcal{F}}(0,T;\mathbb{R}^N)$, moreover, this solution is
unique up to indistinguishability for $Y$, $K$ and equality $d\langle
X,X\rangle_t$ $\times\mathbb{P}$-a.s. for $Z$.
\end{thm}
\subsection{Proof of Uniqueness}
\indent In this section, we first suppose that  solutions of the RBSDE exist, then we prove that they are unique, almost surely.\\[2mm]
\noindent {\bf Proof.}
Suppose $\xi \in L^2(\mathcal{F}_T)$, $f$ satisfies \eqref{Lipch}, \eqref{c''} and \eqref{Con_f} and $G$ satisfies \eqref{Con_g}.  Let $(V^{(1)}, Z^{(1)},K^{(1)})$ and $(V^{(2)}, Z^{(2)},K^{(2)})$ be two solutions of the RBSDE, that is, both $(V^{(1)}, Z^{(1)},K^{(1)})$ and $(V^{(2)}, Z^{(2)},K^{(2)})$ satisfy i) - iii), $V^{(1)}, V^{(2)} \in L^2_{\mathcal{F}}(0,T;\mathbb{R})$, $K^{(1)}_T,K^{(2)}_T\in L^2(\mathcal{F}_T)$ and $Z^{(1)}, Z^{(2)} \in P^2_{\mathcal{F}}(0,T;\mathbb{R}^N)$. Applying the product rule to $|V_t^{(1)}- V_t^{(2)}|$, we have
\begin{align}\label{v1v22}
\nonumber
&|V_t^{(1)} - V_t^{(2)}|^2 \\
\nonumber
& = -2 \int_t^T (V_{u-}^{(1)} - V_{u-}^{(2)}) d(V_{u}^{(1)} - V_{u}^{(2)})  - \sum_{t \leq u \leq T} \Delta (V_u^{(1)}-V_u^{(2)})\Delta (V_u^{(1)}-V_u^{(2)}) \\
\nonumber
&= -2 \int_t^T (V^{(1)}_{u}- V_{u}^{(2)} ) [f(u,V_u^{(2)}, Z_u^{(2)}) - f(u,V_u^{(1)}, Z_u^{(1)})] du \\
\nonumber
& \quad -2 \int_t^T (V^{(1)}_{u}- V_{u}^{(2)}) dK_u^{(2)}  + 2 \int_t^T (V^{(1)}_{u}- V_{u}^{(2)} ) dK_u^{(1)} \\
\nonumber
& \quad -2 \int_t^T  (V^{(1)}_{u-}- V_{u-}^{(2)})  (Z_{u}^{(1)}- Z_{u}^{(2)})' dM_u \\
& \quad - \sum_{t \leq u \leq T} \Delta (V_u^{(1)}-V_u^{(2)})\Delta (V_u^{(1)}-V_u^{(2)}).
\end{align}
We derive
\begin{align}\label{deltav1v2}
\nonumber
&\sum_{t\leq u \leq T}  \Delta (V_u^{(1)}-V_u^{(2)}  ) \Delta (V_u^{(1)}-V_u^{(2)} ) \\
\nonumber
& = \sum_{t\leq u \leq T}( ( Z_{u}^{(1)}-Z_{u}^{(2)})' \Delta X_u)(  ( Z_{u}^{(1)}-Z_{u}^{(2)})'\Delta X_u ) \\
\nonumber
& = \sum_{t\leq u \leq T} (Z_{u}^{(1)}-Z_{u}^{(2)} )' \Delta X_u \Delta X_u' (Z_{u}^{(1)}-Z_{u}^{(2)}  ) \\
\nonumber
& =  \int_t^T ( Z_{u}^{(1)}-Z_{u}^2 )'  (dL_u + d\lel X,X\rir_u) (Z_{u}^{(1)}-Z_{u}^{(2)} )\\
& =  \int_t^T (Z_{u}^{(1)}-Z_{u}^{(2)})' dL_u(Z_{u}^{(1)}-Z_{u}^{(2)} ) + \int_t^T \|Z_{u}^{(1)}-Z_{u}^{(2)} \|_{X_u}^2 du.
\end{align}
From ii) and iii), we know
\begin{align}\label{k1k2}
\nonumber&  - \int_t^T (V^{(1)}_{u}- V_{u}^{(2)}) dK_u^{(2)}  +  \int_t^T (V^{(1)}_{u}- V_{u}^{(2)} ) dK_u^{(1)}\\
\nonumber& = - \int_t^T (V_u^{(1)} - G_u) dK^{(2)}_u + \int_t^T (V_u^{(2)} - G_u) dK_u^{(2)}\\
\nonumber& + \int_t^T (V_u^{(1)} - G_u) dK^{(1)}_u - \int_t^T (V_u^{(2)} - G_u) dK_u^{(1)}\\
\nonumber
& = -\int_t^T (V_u^{(1)} - G_u) dK^{(2)}_u - \int_t^T (V_u^{(2)} - G_u) dK_u^{(1)}\\
&\leq 0.
\end{align}
Therefore, writing $c=\max\{c',c''\}$, by \eqref{v1v22}, \eqref{deltav1v2} and \eqref{k1k2} using the Lipschitz condition, we deduce for any $t\in[0,T],$
\begin{align}\label{star}
\nonumber
&  E \left[ |V_t^{(1)}- V_t^{(2)}|^2 \right]+ E \left[ \int_t^T \|Z_u^{(1)} - Z_u^{(2)}\|^2_{X_u} du \right]\\
\nonumber
& \leq 2 E \left[ \int_t^T |(V^{(1)}_{u}- V_{u}^{(2)} ) (f(u,V_u^{(2)}, Z_u^{(2)}) - f(u,V_u^{(1)}, Z_u^{(1)}))| du \right] \\
\nonumber
&\leq E \left[  2 \int_t^T c (|V_u^{(1)}-V_u^{(2)}|^2+ |V_u^{(1)}-V_u^{(2)}| \cdot\|Z_{u}^{(1)}- Z_{u}^{(2)}\|_{X_u}) du  \right] \\
& \leq E \left[ (2c+2c^2) \int_t^T |V_u^{(1)}-V_u^{(2)}|^2 du + \frac{1}{2}\int_t^T  \|Z_{u}^{(1)}- Z_{u}^{(2)}\|^2_{X_u} du  \right].
\end{align}
That is,
\begin{align*}
E \left[ |V_t^{(1)}- V_t^{(2)}|^2 \right] \leq (2c+2c^2)  E \left[ \int_t^T |V_u^{(1)}-V_u^{(2)}|^2 du \right].
\end{align*}
From Gronwall's lemma, we know $E \left[ |V_t^{(1)}- V_t^{(2)}|^2 \right]=0$ for any $t\in [0,T]$. So for each $t \in [0,T]$, $V_t^{(1)}- V_t^{(2)} =0$, a.s. Since $V^{(1)}$ and $V^{(2)}$ are RCLL, it follows from Lemma \ref{indistinguish} that
$$P( V_t^{(1)}= V_t^{(2)}, \text{ for any } t\in
[0,T])=1.$$
 Also,
$$E \left[ \int_0^T |V_u^{(1)}-V_u^{(2)}|^2 du \right]= \int_0^T E \left[|V_u^{(1)}-V_u^{(2)}|^2 \right]du=0.$$
By (\ref{star}), we obtain
$$E \left[ \int_0^T \|Z_u^{(1)} - Z_u^{(2)}\|^2_{X_u} du \right]=0.$$
Hence $Z_t^{(1)}=Z_t^{(2)}$, $d\lel X,X\rir_t \times \mathbb{P}$-a.s., and from Lemma \ref{Z2}, we derive for any $t\in[0,T]$, $\int_t^T   (Z_{u}^{(1)}-Z_{u}^{(2)})' dM_u=0$, a.s. Using i), we have for any $t\in[0,T],$
\begin{align*}
V_t^{(1)}-V_t^{(2)} & = \int_t^T (f(u, V_u^{(1)}, Z_u^{(1)})-f(u, V_u^{(2)}, Z_u^{(2)})) du \\
&  + (K_T^{(1)}-K_T^{(2)})-(K_t^{(1)}-K_t^{(2)}) - \int_t^T   (Z_{u}^{(1)}-Z_{u}^{(2)})' dM_u.
\end{align*}
Set $t=0$, noticing $K_0^{(1)}=K_0^{(2)}=0$, we deduce
$$\begin{array}{ll}
 |K_T^{(1)}-K_T^{(2)}| \\[2mm]
  \leq |V_0^{(1)}-V_0^{(2)}|+|\int_0^T  (f(u, V_u^{(1)}, Z_u^{(1)})-f(u, V_u^{(2)},
  Z_u^{(2)}))du | \\[2mm]
  + |\int_0^T   (Z_{u}^{(1)}-Z_{u}^{(2)})' dM_u|\\[2mm]
  \leq  \int_0^T c (|V_u^{(1)}-V_u^{(2)}| + \|Z_u^{(1)}-Z_u^{(2)}\|_{X_u}) du \\[2mm]
  =0, ~ \text{ a.s.}
\end{array}$$
Then, similarly, we conclude for any $t\in[0,T]$, $K_t^{(1)}=K_t^{(2)}$, a.s. Since $K$ is continuous, we derive $$P(K_t^{(1)}=K_t^{(2)}, \text{ for any } t \in [0,T])=1.$$
$\mbox{} \hfill \Box$
\subsection{Proof of Existence}
\indent Following \cite{KKPPQ}, in the case of RBSDE driven by a Brownian motion, we proceed with the proof of existence using approximation via penalization.\\
\noindent {\bf Proof of existence.} Set $c=\max\{c',c''\}$.
 For each $n$, consider the following BSDE driven by the Markov chain:
\begin{equation}\label{V_n}
V_t^n = \xi + \int_t^T f(u, V_u^n, Z_u^n) du + n \int_t^T (V_u^n -G_u)^- du - \int_t^T  (Z_{u}^n)' dM_u.
\end{equation}
For $(u,v,z)\in [0,T]\times \mathbb{R}\times \mathbb{R}^N$, define a map:
\[f_n(u, v, z) = f(u, v, z) + n  (v -G_u)^-.\]
For any $u \in [0,T]$ and $(v_1,z_1),(v_2,z_2) \in \mathbb{R}\times \mathbb{R}^N$, we have
\begin{align}\label{fnlip}
\nonumber
&|f_n(u, v_1, z_1)-f_n(u, v_2, z_2)|\\[2mm]
\nonumber
&\leq |f(u, v_1, z_1)-f(u, v_2, z_2)|+n| (v_1 -G_u)^- -(v_2 -G_u)^-|\\[2mm]
\nonumber
&\leq c'|v_1-v_2|+c''\|z_1-z_2\|_{X_u}
+ n| v_1 -v_2|\\[2mm]
&\leq (c'+n)|v_1-v_2|+c''\|z_1-z_2\|_{X_u}.
\end{align}
So $f_n$ is a Lipschitz continuous function in $v$ and $z$. Hence by Lemma \ref{existence}, there exists a unique pair $(V^n, Z^n)\in L^2_{\mathcal{F}}(0,T;\mathbb{R})\times P^2_{\mathcal{F}}(0,T;\mathbb{R}^N)$ which satisfies \eqref{V_n}.
We define:
\[K_t^n = n \int_0^t (V_u^n - G_u)^- du, \quad 0 \leq t \leq T. \]
\begin{lemma} \label{lemma_ito_V}
\begin{align*}
|V_t^n|^2 & = |\xi|^2 + 2 \int_t^T V_u^n f(u, V_u^n, Z_u^n) du +2\int_t^T V_u^n dK_u^n \\
& -2 \int_t^T V_{u-}^n (Z_{u}^n)'dM_u- \int_t^T (Z_{u}^n)' dL_{u} Z_{u}^n - \int_t^T \|Z_u^n\|^2_{X_u} du.
\end{align*}
\end{lemma}
Similar calculations as in \eqref{deltav1v2} yield the result.
Here, we establish for $ (V^n, Z^n, K^n)$ a priori estimates which are independent of $n$.
\begin{lemma}\label{estimateAO}
There exists a constant $C_0>0$, such that for any $n \in \mathbb{N}$:
\begin{equation*}
\sup_{0\leq t \leq T} E [|V_t^n|^2] + E \left[ \int_0^T \|Z_t^n\|^2_{X_t}dt \right]+E [|K_T^n|^2] \leq C_0.
\end{equation*}
\end{lemma}
\begin{proof}
Let $\beta >0$ be an arbitrary constant. Since
\begin{align*}&E \left[ \int_t^T e^{\beta u} (V_u^n -G_u) dK_u^n\right]\\
&= E \left[\int_t^T n e^{\beta u} ((V_u^n -G_u)^+ (V_u^n -G_u)^-  - ((V_u^n -G_u)^-)^2) du \right]\\ & \leq 0,\end{align*}
 we use Lemma \ref{lemma_ito_V} to derive, for any $t \in [0,T]$,
\begin{align*}
&E \left[ e^{\beta t} |V_t^n|^2 \right]+ E \left[ \int_t^T \beta |V_u^n|^2 e^{\beta u} du \right]  + E \left[ \int_t^T e^{\beta u}\|Z_u^n\|^2_{X_u} du \right] \\
& = E \left[ e^{\beta T} |\xi|^2 \right] + 2 E \left[ \int_t^T e^{\beta u} V_u^n f(u, V_u^n, Z_u^n) du \right]+ 2 E \left[ \int_t^T e^{\beta u} V_u^n dK_u^n\right]\\
& \leq  E [e^{\beta T}|\xi|^2] + 2 E\left[ \int_t^T e^{\beta u} (|f(u, 0, 0)| + c |V_u^n| + c \|Z^n_u\|_{X_u}) |V_u^n|du\right] \\
& + 2 E  \left[ \int_t^T  e^{\beta u} G_u dK_u^n \right] 
\end{align*}
\begin{align}\label{apriory1}
\nonumber
& \leq E [e^{\beta T}|\xi|^2] + E \left[ \int_t^Te^{\beta u} |f(u, 0, 0)|^2 du \right]  + (1+ 2c+3c^2) E \left[\int_t^T e^{\beta u}|V_u^n|^2  du \right] \\
\nonumber
& + \frac{1}{3} E \left[ \int_t^T e^{\beta u}\|Z_u^n\|^2_{X_u} du\right]+2 e^{\beta T}E \left[K_T^n \sup_{0 \leq t\leq T} (G_t^+)  \right]  \\
\nonumber
& \leq E [e^{\beta T}|\xi|^2] + E \left[ \int_t^Te^{\beta u} |f(u, 0, 0)|^2 du \right] + (1+ 2c+3c^2) E \left[ \int_t^Te^{\beta u} |V_u^n|^2 du \right] \\
& + \frac{1}{3} E \left[\int_t^T e^{\beta u}\|Z_u^n\|^2_{X_u} du\right]+ \frac{e^{2\beta T}}{\alpha}E \left[ \sup_{0 \leq t \leq T} (G_t^+)^2\right] + \alpha E \left[ (K^n_T)^2\right],
\end{align}
where $\alpha>0$ is an arbitrary constant. Therefore, there exists a constant $C_1 >0$ such that for any $t \in [0,T]$,
\begin{align}\label{apriory2}
\nonumber
& E \left[e^{\beta t} |V_t^n|^2 \right]  +E \left[ \int_t^T \beta |V_u^n|^2 e^{\beta u} du \right] + \frac{2}{3} E \left[ \int_t^T e^{\beta u}\|Z_u^n\|^2_{X_u} du \right]  \\
& \leq C_1(1+ E \left[ \int_t^T e^{\beta u}|V_u^n|^2 du \right]) + \frac{e^{2\beta T}}{\alpha}E \left[ \sup_{0 \leq t \leq T} (G_t^+)^2\right] + \alpha E \left[ (K^n_T)^2\right].
\end{align}
We now give an estimate for $ E \left[ (K^n_T)^2\right]$. From \eqref{V_n}, we have
\[K_T^n  = V_0^n - \xi - \int_0^T f(u, V_u^n, Z_u^n) du + \int_0^T (Z_{u}^n)' dM_u.\]
Then
\begin{align*}
& E \left[ |K_T^n |^2\right] \\
& \leq 4 E \left[  |V_0^n|^2 + |\xi|^2 + |\int_0^T f(u, V_u^n, Z_u^n) du|^2 + |\int_0^T (Z_{u}^{n})'dM_u|^2 \right]\\
& \leq 4 E \left[ |V_0^n|^2 + |\xi|^2\right] \\
& + 4 T E \left[ \int_0^T (|f(u,0,0)|+ c |V_u^n| + c \|Z_u^n\|_{X_u})^2du\right] + 4 E \left[\int_0^T \|Z_u^n\|^2_{X_u} du\right]\\
& \quad \quad \text{(the last integral is obtained using Lemma \ref{Z2})} \\
& \leq 4 \left( E[|\xi|^2] + |V_0^n|^2 + 3T E \left[\int_0^T (|f(u,0,0)|^2+c^2 |V_u^n|^2 + c^2\|Z_u^n\|^2_{X_u}) du \right]\right)\\
& + 4 E \left[ \int_0^T \|Z_u^n\|^2_{X_u} du\right].
\end{align*}
So, there is a constant $C_2>C_1$ such that
\begin{equation}\label{estimate_K}
E [(K_T^n)^2] \leq C_2 \left( 1+ |V_0^n|^2] + E \left[ \int_0^T (|V_u^n|^2 + \| Z_u^n\|^2_{X_u}) du\right]\right).
\end{equation}
Therefore, in \eqref{apriory2}, set $\alpha = 1/3C_2$ to obtain
\begin{align*}
&E[e^{\beta t}|V_t^n|^2] +  E \left[ \int_t^T \beta |V_u^n|^2 e^{\beta u} du \right]+\frac{2}{3} E \left[ \int_t^T e^{\beta u}\|Z_u^n\|^2_{X_u}du \right]\\
 & \leq C_2 (1+ E \left[ \int_t^T e^{\beta u} |V_u^n|^2 du \right]) + 3C_2 e^{2\beta T}E \left[  \sup_{0 \leq t \leq T} (G_t^+)^2 \right] \\
&  +  \frac{1}{3} \left(1 + \sup_{0 \leq t \leq T} E \left[e^{\beta u}|V_t^n|^2 + \int_0^T e^{\beta u} (|V_u^n|^2 + \|Z_u^n\|^2_{X_u})du \right]\right).
\end{align*}
Taking the supremum over $t$, we know
\begin{align*}\label{grown}
& \frac{2}{3} \sup_{0\leq t \leq T}E[e^{\beta t} |V_t^n|^2] + (\beta - C_2 -\dfrac{1}{3}) E \left[ \int_0^T e^{\beta u}|V_u^n|^2 du\right] +\frac{1}{3} E \left[ \int_0^T e^{\beta u} \|Z_u^n\|^2_{X_u} du\right] \\
& \leq  C_2+\frac{1}{3} +  3C_2 e^{2\beta T}E \left[  \sup_{0 \leq t \leq T} (G_t^+)^2 \right].
\end{align*}
Set $\beta = C_2 +\dfrac{1}{3}$. Then, there are two constants $C_3>0$ and $C_4>0$ such that
\[\sup_{0\leq t \leq T}E[e^{\beta t} |V_t^n|^2] \leq C_3\]
and
\[E \left[ \int_0^T e^{\beta u} \|Z_u^n\|^2_{X_u} du\right] \leq C_4.\]
Hence, from \eqref{estimate_K}, we derive
\begin{equation}\label{C_Kn}
 E [|K_T^n|^2] \leq C_5,
\end{equation}
for some constant $C_5>0$. Therefore, there exists a constant $C_0>0$ such that for any $n \in \mathbb{N}$,
\begin{equation*}
\sup_{0\leq t \leq T} E [|V_t^n|^2] + E \left[ \int_0^T \|Z_u^n\|^2_{X_u}du \right]+E [|K_T^n|^2] \leq C_0.
\end{equation*}
\end{proof}
We prove the following:
\begin{lemma}\label{lemma_sup}
For any $n$, there is a constant $C>0$ such that for any $n \in \mathbb{N}$,
\begin{equation*}
E \left[ \sup_{0\leq t \leq T} |V_t^n|^2\right] < C.
\end{equation*}
\end{lemma}
\begin{proof}
We know for any $n \in \mathbb{N}$,
\begin{align*}
|V_t^n|^2 & \leq 4|\xi|^2 + 4|\int_t^T f(u, V_u^n, Z_u^n) du|^2 + 4|K_T^n|^2 + 4|\int_t^T( Z_{u}^n)' dM_u|^2 \\
& \leq 4|\xi|^2 + 12T \int_t^T (|f(u,0,0)|^2 + c^2 |V_u^n|^2 + c^2 \|Z_u^n\|^2_{X_u}) du \\
& + 4 |K_T^n|^2 + 4|\int_t^T  (Z_{u}^n)' dM_u|^2.
\end{align*}
Taking the supremum over $t$, we deduce
\begin{align}\label{2star}
\nonumber
 \sup_{0\leq t \leq T} |V_t^n|^2 & \leq 4 |\xi|^2 + 12 T \int_0^T |f(u,0,0)|^2  du \\
\nonumber
& + 12Tc^2\int_0^T |V_u^n|^2 du + 12Tc^2 \int_0^T \|Z_u^n\|^2_{X_u} du \\
& +4 |K_T^n|^2+ 4 \sup_{0\leq t \leq T}|\int_t^T (Z_{u}^n)'dM_u|^2.
\end{align}
Using Doob's inequality and Lemma \ref{Z2}, we obtain
\begin{align*}
 &E \left[ \sup_{0\leq t \leq T}|\int_t^T ( Z_{u}^n)' dM_u|^2 \right]\\
 & = E \left[ \sup_{0\leq t \leq T}|\int_0^T ( Z_{u}^n)' dM_u-\int_0^t ( Z_{u}^n)' dM_u|^2 \right] \\
  & \leq 2 E \left[ |\int_0^T ( Z_{u}^n)' dM_u|^2 \right]+2 E \left[ \sup_{0\leq t \leq T}|\int_0^t ( Z_{u}^n)' dM_u|^2 \right] \\
& \leq 10 E \left[ |\int_0^T( Z_{u}^n)' dM_u|^2\right]=10 E \left[ \int_0^T \|Z_u^n\|^2_{X_u}du\right].
\end{align*}
Also,
\begin{align*}
 E \left[ \int_0^T |V_u^n|^2 du \right]  = \int_0^T E [|V_u^n|^2] du  \leq T \sup_{0\leq t \leq T} E [|V_t^n|^2].
\end{align*}
By Lemma \ref{estimateAO}, there is a constant $C>0$ such that the result holds.
\end{proof}
Now, we prove:
\begin{lemma}\label{limit_V}
There is process $\{V_t,~ t \in [0,T]\}$ such that  $V\in L^2_{\mathcal{F}}(0,T;\mathbb{R})$,
\[E \left[ \int_0^T (V_t - V_t^n)^2 dt \right] \rightarrow 0,~~ \text{ as } n \rightarrow \infty,\]
and
\[E \left[ \sup_{0\leq t\leq T} |V_t|^2 \right]  \leq C.\]
\end{lemma}
\begin{proof} Since $f_n(\cdot, \cdot, \cdot)$ is increasing in $n$, that is, for any $(t,y,z)\in [0,T]\times \mathbb{R} \times \mathbb{R}^N$, $n\in\mathbb{N}$,
\[f_n(t, y, z) \leq f_{n+1}(t, y, z),\]
moreover, $f_n$ satisfies \eqref{fnlip} and the constant $c''$ satiisfies \eqref{c''},
by Lemma \ref{CT} we derive for any $n \in \mathbb{R}^N$,
$$P(V_t^n \leq V_t^{n+1},  ~\text{for any}~ t \in [0,T])=1.$$
 That is, for any $n \in \mathbb{N}$, there exists a subset $B_n \subseteq \Omega$ and $\hat{B} \subseteq \Omega$ such that $\hat{B} = \bigcap\limits_{n=1}^{\infty} B_n$, $P(\hat{B}) =1$ and for any $\omega \in \hat{B}$, $V_t^n(\omega) \leq V_t^{n+1}(\omega)$, $t \in [0,T]$.
 For any $\omega \in \Omega$, define:
\[V_t(\omega) = \sup_{n \in \mathbb{N}} V_t^n (\omega), ~ \quad t \in [0,T].\]
So $ P(V_t^n \uparrow V_t$, ~$t \in [0,T])=1$. Therefore,
 \[P (\mathbb{I}_{\{V_t > 0\}}|V_t^n|\uparrow  \mathbb{I}_{\{V_t > 0\}} |V_t|, ~~t \in [0,T]) =1\]
and
 \[P (\mathbb{I}_{\{V_t \leq 0\}}|V_t^n| \downarrow  \mathbb{I}_{\{V_t \leq 0\}} |V_t|, ~~ t \in [0,T]) =1.\]
By Levi's Lemma and Lemma \ref{lemma_sup}, we deduce
\[E \left[ \int_0^T |V_t|^2 dt \right] = \lim_{n \rightarrow \infty } E \left[ \int_0^T |V_t^n|^2 dt \right] \leq \lim_{n\rightarrow \infty}(\sup_{0 \leq t \leq T} E [|V_t^n|^2] T) \leq CT.\]
Then $V\in L^2_{\mathcal{F}}(0,T;\mathbb{R})$ and $|V_t| < \infty$, a.e, a.s. So $V_t^n - V_t \uparrow 0,$ a.e, a.s. Again, by Levi's Lemma, we have
\[E \left[ \int_0^T |V_t^n -V_t|^2 dt \right]\rightarrow 0, \text{  as  } n \rightarrow \infty.\]
Since $\{\sup\limits_{0 \leq t \leq T} V_t^n, n\in \mathbb{N}\}$ is also an increasing sequence, we know there exists a random variable $H$ such that for any $\omega \in \Omega$:
\[\sup_{n \in \mathbb{N}}\sup_{0\leq t \leq T} V_t^n (\omega) = H(\omega),\]
so
\[\sup_{0\leq t \leq T} V_t^n \uparrow H,  \text{  a.s.}\]
Also, by Levi's lemma, we obtain
\[\lim_{n \rightarrow \infty} E[\sup_{0 \leq t \leq T}|V_t^n|^2] = E [|H|^2].\]
By Lemma \ref{lemma_sup}, we deduce that $E [|H|^2] \leq C$. Hence,
\begin{align*}
 E \left[ \sup_{0 \leq t \leq T} |V_t|^2\right] & = E [\sup_{0 \leq t \leq T} (\lim_{n \rightarrow \infty} |V_t^n|^2)] \\
&  \leq E \left[ \sup_{0 \leq t \leq T} (\lim_{n \rightarrow \infty} (\sup_{0 \leq t \leq T} |V_t^n|^2))\right]\\
& \leq E [\sup_{0\leq t \leq T }|H|^2]  = E [|H|^2]\leq C.
\end{align*}
Hence, we proved Lemma \ref{limit_V}.
\end{proof}
Now, consider the same set $\hat{B}$ in the proof of Lemma \ref{limit_V}. Also, by Lemma \ref{limit_V} $\sup\limits_{0\leq t \leq T}|V_t| < \infty$, a.s., that is, there is a subset $\bar{B} \subseteq \Omega$ such that for any $\omega \in \bar{B}$, $|V_t(\omega)| < \infty$ for any $t\in [0,T]$ and $P(\bar{B})=1$. Then, for $\omega \in \hat{B}\cap\bar{B}$,
\[V_t^n(\omega)-V_t(\omega) \uparrow 0 , \quad t \in [0,T].\]  By Lemma \ref{polya}, we derive for any $\omega \in \hat{B}\cap \bar{B}$,
\[\lim_{n\rightarrow \infty} \sup_{0\leq t \leq T} |V_t^n(\omega) - V_t(\omega)|^2 =0 .\]
Since $P(\hat{B}\cap\bar{B}) =1$, it follows that:
\[\lim_{n\rightarrow \infty} E \left[ \sup_{0 \leq t \leq T} |V_t^n-V_t|^2\right] =0.\]
Hence, $\{V^n\}_{n\in \mathbb{N}}$ is a uniform Cauchy sequence, that is:
\begin{equation} \label{cauchyV}
E \left[ \sup_{0 \leq t \leq T} |V_t^n - V_t^p|^2 \right] \rightarrow 0, \text{ as } n,p \rightarrow \infty.
\end{equation}
 Now, using Lemma \ref{lemma_ito_V} for $|V_t^n - V_t^p|^2$, and taking the expectation, gives:
\begin{align*}
& E[|V_t^n - V_t^p|^2] + E \left[ \int_t^T \|Z_u^n - Z_u^p\|_{X_u}^2 du\right]\\
& =  2 E \left[ \int_t^T (f(u, V_u^n, Z_u^n) - f(u,V_u^p,Z_u^p)) (V_u^n - V_u^p) du\right] \\
&  + 2 E \left[ \int_t^T (V_u^n - V_u^p) d(K_u^n -K_u^p)\right].
\end{align*}
Noting that $$dK_u^n = n (V_u^n-G_u)^- du$$ then
\begin{align*}
& E \left [ \int_t^T (V_u^n -V_u^p) d(K_u^n -K_u^p) \right] \\
& = E \left[ \int_t^T (V_u^n -G_u) dK_u^n \right] -  E \left[ \int_t^T (V_u^n -G_u) dK_u^p \right] \\
& - E \left[ \int_t^T (V_u^p -G_u) dK_u^n \right] + E \left[ \int_t^T (V_u^p -G_u) dK_u^p \right]\\
& \leq E \left[ \int_t^T (V_u^n -G_u)^- dK_u^p \right] + E \left[ \int_t^T (V_u^p-G_u)^- dK_u^n\right].
\end{align*}
Thus,
\begin{align*}
& E[|V_t^n -  V_t^p|^2] + E \left[ \int_t^T \|Z_u^n - Z_u^p\|_{X_u}^2 du\right]\\
& \leq 2c E \left[ \int_t^T \left( |V_u^n-V_u^p|^2 + |V_u^n-V_u^p|\cdot \|Z_u^n - Z_u^p\|_{X_u}\right) du\right] \\
& + 2E \left[ \int_t^T (V_u^n - G_u)^- dK_u^p \right] + 2 E \left[ \int_t^T (V_u^p -G_u)^- dK_u^n \right]\\
\nonumber
 & \leq (2c+2c^2)  E \left[ \int_t^T  |V_u^n-V_u^p|^2  du \right] + \frac{1}{2} E \left[ \int_t^T \|Z_u^n - Z_u^p\|_{X_u}^2 du \right] \\
 &\quad + 2E \left[ \int_t^T (V_u^n - G_u)^- dK_u^p \right] + 2 E \left[ \int_t^T (V_u^p -G_u)^- dK_u^n \right].
\end{align*}
That is,
\begin{align}\label{Z_n_p}
\nonumber
& E \left[ \int_t^T \|Z_u^n - Z_u^p\|_{X_u}^2 du\right] \\
\nonumber
& \leq (4c+4c^2)  E \left[ \int_t^T  |V_u^n-V_u^p|^2  du \right] \\
& + 4E \left[ \int_t^T (V_u^n - G_u)^- dK_u^p \right] + 4 E \left[ \int_t^T (V_u^p -G_u)^- dK_u^n \right].
\end{align}
\begin{lemma}\label{V_G}
 \[E \left[ \sup_{0 \leq t \leq T} |(V_t^n - G_t)^-|^2 \right] \rightarrow 0, \text{ as } n\rightarrow \infty.\]
\end{lemma}
\begin{proof}
As $V_t^n \geq V_t^0$, replace $G_t$ by $G_t \vee V^0_t$. Since
$E [\sup\limits_{0 \leq t\leq T} |V^0_t|^2] < \infty,$ we have \[E [\sup_{0 \leq t\leq T} |G_t \vee V^0_t|^2] < \infty.\]We shall compare $V_t$ and $G_t$. For $n \in \mathbb{N}$, consider the following BSDE for the Markov chain:
\[\tilde{V}_t^n = \xi + \int_t^T F_n(u, \tilde{V}_u^n, \tilde{Z}_u^n)du - \int_t^T( \tilde{Z}_{u}^n)' dM_u,\]
where $F_n(u, v, z) = f(u, Y_u^n, Z_u^n) + n (G_u - v)$. Then, by Lemma \ref{existence}, for each $n \in \mathbb{N}$, there exists a unique solution $(\tilde{V}^n,\tilde{Z}^n)\in L^2_{\mathcal{F}}(0,T;\mathbb{R})\times P^2_{\mathcal{F}}(0,T;\mathbb{R}^N)$ to the above BSDE. As $ (G_u - v) \leq (v-G_u)^- $ for any $u \in [0,T]$, it follows that $F_n(u, v, z) \leq f_n(u,v,z)$ for any $u\in[0,T]$, $(v, z) \in\mathbb{R}\times \mathbb{R}^N$. Hence, from Lemma \ref{CT}, $$P(\tilde{V}_t^n \leq V_t^n, \text{ for any } t\in [0,T])=1.$$
Let $\tau \in [0,T]$ be a stopping time. By Lemma \ref{BSDEST}, the following BSDE for the Markov chain with stopping time $\tau$
\[\tilde{V}_{\tau}^n = \xi + \int_{\tau}^T F_n(u, \tilde{V}_u^n, \tilde{Z}_u^n)du - \int_{\tau}^T( \tilde{Z}_{u}^n)' dM_u\]
has a unique solution.
  Then, applying Ito's formula to $e^{-n \tau} \tilde{V}_{\tau}^n$, we have
\begin{align*}
e^{-nT}\xi - e^{-n\tau} \tilde{V}_{\tau}^n & = \int_{\tau}^T e^{-nu} \left( - f(u, V_u^n, Z_u^n) - n (G_u - \tilde{V}_u^n)\right) du \\
& \quad + \int_{\tau}^T e^{-nu}( \tilde{Z}_{u}^n)'dM_u -n \int_t^T \tilde{V}_u^n e^{-nu} du.
\end{align*}
Rearranging and taking the expectation given $\mathcal{F}_{\tau}$, we derive
\begin{equation}\label{vto}
\tilde{V}^n_{\tau} = E \left[ e^{-n(T - \tau)}\xi + \int_{\tau}^T e^{-n(u-\tau)} f(u,V_u^n, Z_u^n) du + n \int_{\tau}^T e^{-n(u-\tau)}G_u du| \mathcal{F}_\tau\right].
\end{equation}
It is easy to see that as $n \rightarrow \infty$,
\[e^{-n(T - \tau)}\xi +  n \int_{\tau}^T e^{-n(u-\tau)}G_u du \rightarrow \xi 1_{\{\tau=T\}} + G_{\tau}1_{\{\tau<T\}},\]
a.s, and in mean square. So
\begin{equation}\label{E1_vto}
E \left[ e^{-n(T - \tau)}\xi +  n \int_{\tau}^T e^{-n(s-\tau)}G_u du| \mathcal{F}_{\tau}\right] \rightarrow E\left[\xi 1_{\{\tau=T\}} + G_{\tau}1_{\{\tau<T\}}|\mathcal{F}_{\tau} \right]
\end{equation}
in mean square.
 Also, by Holder's inequality, we know
\begin{align*}
\lvert  \int_{\tau}^T e^{-n(u-\tau)} f(u,V_u^n, Z_u^n) du \lvert & \leq \left( \int_{\tau}^T  e^{-2n(u-\tau)} du \right)^{1/2} \left( \int_{\tau}^T |f(u,V_u^n, Z_u^n)|^2 du\right)^{1/2}\\
& \leq  \left( \int_{\tau}^T  e^{-2n(u-\tau)} du \right)^{1/2} \left( \int_{0}^T |f(u,V_u^n, Z_u^n)|^2 du\right)^{1/2} \\
& \leq (\frac{1}{2n} (1- e^{-2n(T-\tau)}))^{1/2}  \left( \int_{0}^T |f(u,V_u^n, Z_u^n)|^2 du\right)^{1/2} \\
& \leq \frac{1}{\sqrt{2n}}  \left( \int_{0}^T |f(u,V_u^n, Z_u^n)|^2 du\right)^{1/2}.
\end{align*}
Hence,
\begin{equation}\label{E2_vto}
E \left[ \int_{\tau}^T e^{-n(u-\tau)} f(u,V_u^n, Z_u^n) du | \mathcal{F}_{\tau}\right] \rightarrow 0
\end{equation}
in mean square, as $n \rightarrow \infty.$ Therefore, from \eqref{vto}, \eqref{E1_vto} and \eqref{E2_vto},
\[\tilde{V}_{\tau}^n \rightarrow \xi 1_{\{\tau=T\}} + G_{\tau}1_{\{\tau<T\}}\]
in mean square. Since $V_{\tau}^n \leq V_{\tau}$, a.s., and $V_{\tau}^n \geq \tilde{V}_{\tau}^n$, we obtain
\[V_{\tau} \geq \xi 1_{\{\tau=T\}} + G_{\tau}1_{\{\tau<T\}},\]
and it follows that $V_{\tau} \geq G_{\tau}$, a.s that is $(V_{\tau} - G_{\tau})^- =0$, a.s. Therefore, by the Section Theorem (\cite{Della} page 220 or \cite{elliott} Corollary 6.25), we have
$$P ( (V_t - G_t)^- =0, ~~ t\in [0,T]) =1.$$
So
$$P ( (V_t^n - G_t)^- \downarrow 0, ~~ t \in [0,T])=1 .$$
Noting, for a.s. $\omega \in \Omega$,
\begin{align*}
(V_t^n - G_t)^- & = \frac{1}{2} (|V_t^n -G_t|- (V_t^n -G_t))\\
& \leq \frac{1}{2} (|V_t^n -V_t|+ |V_t-G_t|- (V_t^n-V_t) -(V_t -G_t)) \\
& = \frac{1}{2} (|V_t^n-V_t|+ V_t-G_t - (V_t^n -V_t) - (V_t-G_t)) \\
& \leq |V_t^n -V_t|, \text{ for any } t \in [0,T],
\end{align*}
we deduce
\[0 \leq \sup_{0 \leq t \leq T} (V_t^n - G_t)^-  \leq \sup_{0 \leq t \leq T}  |V_t^n -V_t|.\]
Since $\lim\limits_{n \rightarrow \infty}\sup\limits_{0 \leq t \leq T}  |V_t^n -V_t| = 0, \text{ a.s.}$, we obtain
$$ \lim_{n \rightarrow \infty} \sup_{0\leq t \leq T} (V_t^{n} - G_t)^- = 0, \text{ a.s.}$$
As,
\[(V_t^n -G_t)^- \leq (G_t - V_t^0)^+ \leq |G_t| + |V_t^0|,\]
 $E \left[ \sup\limits_{0 \leq t \leq T} G_t^2 \right] <\infty$ and $E \left[ \sup\limits_{0 \leq t \leq T} (V^0_t)^2 \right] <\infty$, then the result follows from the dominated convergence theorem.\\
\end{proof}
\indent Returning to \eqref{Z_n_p}, we have
\begin{align*}
E \left[ \int_t^T (V_u^p - G_u)^- dK_u^n \right] & \leq E \left[ \int_0^T \sup_{0 \leq u\leq T}(V_u^p - G_u)^- dK_u^n \right] \\
& \leq E \left[  \sup_{0 \leq t\leq T}(V_t^p - G_t)^- K^n_T\right] \\
& \leq \left( E \left[  \sup_{0 \leq t\leq T}|(V_t^p - G_t)^-|^2\right]\right)^{1/2}\left( E \left[ |K_T^n|^2\right]\right)^{1/2}.
\end{align*}
Hence, from \eqref{C_Kn} and Lemma \ref{V_G}, we deduce as $n,p \rightarrow \infty$,
\begin{equation} \label{v_g_k_0}
 E \left[ \int_t^T (V_u^p - G_u)^- dK_u^n \right] + E \left[ \int_t^T (V_u^n - G_u)^- dK_u^p \right] \rightarrow 0.
\end{equation}
It follows from \eqref{Z_n_p}, \eqref{v_g_k_0} and Lemma \ref{limit_V} that as $n,p\rightarrow \infty$:
\begin{align} \label{v_z_0} E \left[ \int_0^T  \|Z_u^p - Z_u^n\|^2_{X_u} du \right] \rightarrow 0 .
\end{align}
Consider the factor space of equivalence classes of processes in $P^2_{\mathcal{F}}(0,T;\mathbb{R}^N)$. An equivalence class is just all processes which differ by a null process. On that space the semi norm is actually a norm and so the space is complete. Then there exists a process $Z\in P^2_{\mathcal{F}}(0,T;\mathbb{R}^N)$ such that as $n\rightarrow \infty$,
$ E \left[ \int_0^T  \|Z_u^n- Z_u\|^2_{X_u} du \right] \rightarrow 0 .$
Now
\begin{align*}
& K_t^n - K_t^p \\
& = (K_T^n-K_T^p) + \int_t^T( f(u,V_u^n,Z_u^n) -f(u,V_u^p, Z_u^p))du \\
& - (V_t^n - V_t^p) - \int_t^T( Z_{u}^n-Z_{u}^p )'dM_u
\end{align*}
\begin{align*}
& = (V_0^n-V_0^p) - \int_0^t( f(u, V_u^n, Z_u^n)- f(u, V_u^p, Z_u^p)) du + \int_0^t ( Z_{u}^n - Z_{u}^p)'dM_u \\
& - (V_t^n - V_t^p).
\end{align*}
Using Doob's inequality and Lemma \ref{Z2} on the last equation, we derive:
\begin{align*}\label{kpn2}
\nonumber
&E[\sup_{t\in[0,T]}|K_t^n - K_t^p|^2]\\
\nonumber
& \leq 4 E[|V_0^n-V_0^p|^2 ]+4 E[\sup_{t\in[0,T]}|\int_0^t( f(u, V_u^n, Z_u^n)- f(u, V_u^p, Z_u^p)) du |^2]\\
\nonumber
&  + 4E[\sup_{t\in[0,T]}|\int_0^t ( Z_{u}^n - Z_{u}^p)'dM_u|^2]  + 4E[\sup_{t\in[0,T]}|V_t^n - V_t^p|^2]\\
\nonumber
& \leq 8E[\sup_{t\in[0,T]}|V_t^n - V_t^p|^2]+4 E[(\int_0^T| f(u, V_u^n, Z_u^n)- f(u, V_u^p, Z_u^p)| du )^2]\\
\nonumber
&+ 16E[|\int_0^T ( Z_{u}^n - Z_{u}^p)'dM_u|^2]\\
& \leq 8E[\sup_{t\in[0,T]}|V_t^n - V_t^p|^2]+8c^2T E[\int_0^T| V_u^n- V_u^p|^2 du ]\\
&+(16+8c^2T)E[\int_0^T\| Z_{u}^n - Z_{u}^p)\|^2_{X_u}d u].
\end{align*}
Therefore by \eqref{cauchyV} and \eqref{v_z_0}:
\begin{equation*}
E\left[\sup_{0\leq t\leq T}|K_t^n -K_t^p|^2 \right] \rightarrow 0, \text{ as } n,p \rightarrow \infty.
\end{equation*}
Hence, $\{K^n\}_{n \in \mathbb{N}}$ is a Cauchy sequence which converges uniformly to some limit $K$ in mean square. Since $\{V^n\}_{n \in \mathbb{N}}$ and $\{Z^n\}_{n \in \mathbb{N}}$ are Cauchy sequences which converge to $V$ and $Z$, we know $V, Z, K$ satisfy i). Moreover, $K$ is continuous and increasing. Condition ii) follows from the proof of Lemma \ref{V_G}. Next we prove the remaining part of Condition iii). We know $(V^n,K^n)$ converges uniformly in $t$ to $(V,K)$ in probability. Therefore the measure $dK^n$ converges to $dK$ weakly in probability. It follows that:
\[\int_0^T (V_t^n-G_t)dK_t^n \rightarrow \int_0^T (V_t-G_t) dK_t\]
in probability. Using Lemma \ref{V_G} we deduce that:
\[\int_0^T (V_t - G_t) dK_t \geq 0, \text{ a.s.}\]
However,
\[\int_0^T (V^n_t - G_t) dK^n_t = n \int_0^T (V^n_t - G_t) (V^n_t - G_t)^- dt \leq 0, \quad n \in \mathbb{N}.\]
Hence,
\[\int_0^T (V_t-G_t) dK_t =0,\text{  a.s.}\]
Finally, we conclude that $(V, Z,K)$ solves the RBSDE.~~~~~ $\mbox{} \hfill \Box$
\section{Application to American Options}
\subsection{The Stochastic Discount Function (SDF)}
\indent As in \cite{RE1} and \cite{RE2}, we give the following definition:
\begin{definition}
A stochastic discount process is an adapted stochastic process $\pi =\{\pi_t, t\geq 0 \}$ such that for any asset price process $\{\mathcal{A}_t,t\geq 0\}$,
\[\pi_t\mathcal{A}_t=E[\pi_s\mathcal{A}_s|\mathcal{F}_t].\]
Here, $E$ is expectation with respect to the real world probability $P$.
\end{definition}
We suppose the stochastic discount function is modeled as follows:
\[\pi_t=\exp \left[ -\int_0^t X'_{u-}C_udX_u -\int_0^t D'_uX_udu \right],\]
where $C_u$ is an $N\times N$ matrix and $D_u$ is a vector in $\mathbb{R}^N$ for each $u\geq 0$.\\
The following lemma is Theorem 3.1 in \cite{RE1}.
\begin{lemma}\label{lemma_pi}
\begin{equation*}
d\pi_t = \pi_t [-D'_tX_t+X'_t\sigma_tA_tX_t] dt + \pi_{t_{-}} X'_{t_{-}}\sigma_{t_{-}}dM_t,
\end{equation*}
where $\sigma_t = (\sigma_t^{ij})$ is the $N \times N$ matrix with:
\begin{equation*}
\sigma_t^{ij} = \exp (C_t^{ii}- C_t^{ij}) -1, \quad \quad 1 \leq i,j\leq N.
\end{equation*}
\end{lemma}

Denote by $\Gamma$, the matrix whose components are:
\begin{align*}
\Gamma_t^{ii}&=A_t^{ii}-D_t^{i} \text{ and } \\
\Gamma_{t}^{ij}&=A_{t}^{ij}\exp (C_t^{jj}-C_t^{ji}) \text{ if } i \neq j,
\end{align*}
\subsection{The Market}
\indent We consider a market consisting of $n$ stocks with price process $S^j = \{S^j_t, t \in [0,T]\}$, $j=1,2,\cdots,n$ and a bond with price $B = \{B_t, t \in [0,T] \}$, where  $T <\infty$ will be the maturity time.
Suppose each stock $S^j$ pays, at any time $t\in [0,T]$, a dividend denoted by $\mathcal{D}^j_ t$ and for a vector function $\delta_{j,t} \in \mathbb{R}^N$, it has the form $\mathcal{D}^j_ t = \delta'_{j,t} X_t$. The stock price is the discounted value of all future dividends. It is shown \cite{RE1} that, for any $t \in [0,T]$, $S_t^j$ can be written in the form $S^j_t =s'_{j,t}X_t$  where $s_{j,t} \in \mathbb{R}^N$ is a function satisfying the vector ordinary differential equation:
\begin{equation}\label{stock_ode}
\frac{ds_{j,t}}{dt} + \Gamma'_t s_{j,t} = - \delta_{j,t} ~~ \text{ and } ~ s_{j,t} \rightarrow 0 \text{ as } t \rightarrow 0.
\end{equation}
Note, for each $j=1,2,\cdots,n$, the vector function $s_{j,t} \in \mathbb{R}^N$ is the solution of the ordinary differential equation \eqref{stock_ode}, hence its $i$-th component $s_{j,t}^i$ is continuous on the domain $[0, + \infty)$. Therefore, on the interval $[0,T]$, for each $i$, $s_{j,t}^i$ is bounded. Moreover, we suppose stock prices are strictly positive, hence $s_{j,t}^i$ is strictly positive for each $i$ and $j$. Therefore there is $c_2 >0$ and $c_3 >0$ such that:
\begin{equation} \label{sbound}
c_2 \leq s_{j,t}^i \leq c_3 \quad \text{ for any } i=1,\cdots,N;~j=1,\cdots,n.
\end{equation}
Lemma 6.1 in \cite{RE1} gives the dynamics of the stock prices $S^j$ as:
\begin{lemma}\label{stockS}
\begin{equation*}
S^j_t = S^j_0 + \int_0^t  ((A'_u -\Gamma'_u) s_{j,u})'X_u du - \int_0^t  \delta'_{j,u} X_u  du + \int_0^t  s'_{j,u}dM_u ,
\end{equation*}
$j=1,\cdots,n$ and $t \in [0,T]$.
\end{lemma}
Let $r_t \in \mathbb{R}$ be the interest rate at any time $t \in [0,T]$, so the bond price has the dynamics:
\begin{equation*}
dB_t = r_t B_t dt,
\end{equation*}
It is shown in \cite{RE1} that:
\begin{lemma}
For any $t \in [0,T]$,
\begin{equation*}
r_t =  D'_t X_t - X'_t\sigma_t A_t X_t.
\end{equation*}
\end{lemma}
Hence, the dynamics of the stochastic discount function $\pi$ in Lemma \ref{lemma_pi} becomes:
\begin{equation}\label{pi}
d \pi_t = - \pi_t r_t dt + \pi_{t_{-}} X'_{t_{-}}\sigma_{t_{-}} dM_t ~~ \text{for any}~ t \in [0,T].
\end{equation}
It is known that the market in the presence of a positive discount factor has no arbitrage opportunity.
\subsection{The Self-financing Super-hedging Strategy}
\indent We state the following definitions:
\begin{definition}[Self-financing strategy]
Let $V$ be the portfolio value, $h^0_t \in \mathbb{R}$ the number of bonds $B$ held at time $t$ and let $h_t = (h_t^1,\cdots,h_t^n)'$ with $h^j_t \in \mathbb{R}$ is the number of stocks $S^j$ held at time $t$, $j=1, \cdots,n$. Then
\begin{equation}\label{portfolio}
V_t = h_t^0B_t +\sum\limits_{j=1}^n h_t^jS_t^j, ~~ t \in [0,T].
\end{equation}
Let $K$ be the cumulative consumption process with $K_0= 0$. Then, a self-financing strategy,  is a vector process $(V, h, K)$ such that:
\begin{equation}\label{self_financing}
dV_t = h_t^0 dB_t + \sum_{j=1}^n ( h_t^j dS^j_t + h_t^j d\mathcal{D}^j_t) - dK_t, ~~ t \in [0,T].
\end{equation}
\end{definition}
For American options, the portfolio value should dominate the payoff at any time $t$ to cover any exercise action. This leads to the following definition:
\begin{definition}
Given a payoff process $\{G_t\}$, a self-financing strategy is called a superhedging strategy if:
\[V_t \geq G_t, \quad   t \in [0,T) ~\text{ and  }~~ V_T = G_T.\]
\end{definition}

We shall discuss whether we can find such a strategy. The theory of RBSDEs, driven by Brownian motions, ensures the existence of such strategy in the classical Black-Scholes model. We shall show a similar result for the Markov chain model. \\
\indent It follows from Lemma \eqref{stockS}, \eqref{portfolio} and \eqref{self_financing} that:
\begin{align}\label{dV}
\nonumber
dV_t &  = h_t^0 r_t B_t dt + \sum_{j=1}^n h_t^j X'_t (A'_t - \Gamma'_t)s_{j,t} dt + \sum_{j=1}^n h_{t-}^j s'_{j,t} dM_t - dK_t \\
 \nonumber
 & = r_t V_t dt - r_t (\sum_{j=1}^n h_t^j X_t's_{j,t})dt+ \sum_{j=1}^n h_t^j X'_t (A'_t - \Gamma'_t)s_{j,t} dt + \sum_{j=1}^n h_{t-}^j s'_{j,t} dM_t - dK_t \\
\nonumber
 & = r_tV_t dt + \sum_{j=1}^n h_t^j X_t' (-r_t + (A'_t -\Gamma'_t))s_{j,t}dt + \sum_{j=1}^n h^j_{t-}s'_{j,t} dM_t -dK_t \\
& = r_tV_t dt +  X_t' (-r_t + (A'_t -\Gamma'_t))(\sum_{j=1}^n h_t^j s_{j,t})dt + (\sum_{j=1}^n h^j_{t-}s'_{j,t}) dM_t -dK_t.
\end{align}

Now, consider the
function $f: \mathbb{R} \times \mathbb{R} \times \mathbb{R}^N \rightarrow \mathbb{R}$ such that
\begin{equation}\label{f}
f(t, v, z) = -r_t v - X'_t( -r_t + (A'_t-\Gamma'_t)) z,
\end{equation}
and the RBSDE:
\begin{equation}\label{R}\begin{cases}
 \text{i})& V_t  = G_T + \int_t^T f(u, V_u , Z_u) du + K_T - K_t - \int_t^T Z_{u-}'dM_u ;\\
 \text{ii})& V_t  \geq G_t;\\
 \text{iii})& \{K_t, t \in [0,T]\} \text{ is continuous and increasing},~ K_0=0 \\
   & ~ \text{ and } \int_0^T (V_u - G_u) dK_u = 0.
\end{cases} \end{equation}
\begin{prop}[Lipschitz Condition.] \label{lipf}
Let $f$ be given by \eqref{f}. We suppose that there is a constant $c_1 >0$ such that:
\begin{equation}\label{normGA}
| (A_t - \Gamma_t )X_t|_N \leq c_1,
\end{equation}
for any $t \in [0,T]$.
Then, for $z_1, z_2 \in \mathbb{R}^N$ and for $v_1, v_2 \in \mathbb{R}$, there is a constant $c_6>0$ such that:
\[|f(t, v_1, z_1 ) - f(t, v_2, z_2)| \leq c_6 (\|z_1 - z_2\|_{X_t}+ |v_1 - v_2|),\]
for any $t \in [0,T]$.
\end{prop}
\begin{proof}
The interest rate $r_t$ is, in practice, positive and bounded, so there is $c_4 >0$ such that
\begin{equation}\label{normr}
r_t \leq c_4.
\end{equation}
From \eqref{normr} and \eqref{normGA}, there is a constant $c_5 >0$ such that
\begin{equation} \label{normS}|(-r_t + (A_t - \Gamma_t))X_t|_N \leq c_5.\end{equation}
Now, for $z_1,z_2 \in \mathbb{R}^N$ and $v_1, v_2 \in \mathbb{R}^N$, we have
\begin{align*}
& |f(t, v_1, z_1) - f(t, v_2, z_2)| \\
&= |(v_1-v_2)r_t + X'_t(-r_t + (A'_t - \Gamma'_t))(z_1-z_2)| \\
& \leq |v_1-v_2|r_t + |(-r_t + (A_t - \Gamma_t))X_t |_N \times |z_1-z_2|_N.
\end{align*}
From Lemma \ref{normbound}, there is a constant $\beta>0$ such that $|z_2 - z_1 |_N \leq \sqrt{3\beta} \|z_2 - z_1 \|_{X_t}$. Hence, with \eqref{normS}, there is a constant $c_6$, such that
\[|f(t, v_1, z_1) - f(t, v_2, z_2)| \leq c_6 ( \|z_2 - z_1 \|_{X_t} + |v_1-v_2|).\]
\end{proof}
Therefore, from previous section RBSDE \eqref{R} has a unique solution $(V,Z,$ $K)$ such that $V \in L^2_{\mathcal{F}}(0,T;\mathbb{R})$, $K_T\in L^2(\mathcal{F}_T)$ and $Z \in P^2_{\mathcal{F}}(0,T;\mathbb{R}^N)$
Now, if $(V,Z,K)$ is the unique solution to RBSDE \eqref{R} and if there exists a non-zero vector $h=(h_t^1, \cdots, h_t^n)'$ such that $\sum\limits_{j=1}^n h_t^j s_{j,t} = Z_t$, then  $(V,h,K)$ solves \eqref{dV}. The equation $\sum\limits_{j=1}^n h_t^j s_{j,t} = Z_t$ has a solution $h_t$, $t \in [0,T]$ if $Z_t$ belongs to the linear subspace of $\mathbb{R}^N$ spanned by the vectors $s_{1,t},\cdots,s_{n,t}$, which holds only if $n \geq N$. Moreover, the decomposition of $Z_t$ into a linear combination of $s_{j,t}$'s is unique if $s_{j,t}$'s are linearly independent, in which case, $n$ cannot be greater than $N$, hence $n=N$. This leads to the following proposition:
\begin{prop}\label{complete}
Suppose $f$, in equation \eqref{f} and Proposition \ref{lipf}, satisfies $c_6\|\Psi^{\dagger}_t\|_{N\times N} \sqrt{6m} < 1$. A unique super hedging strategy $(V,h,K)$ exists for the American option with payoff $G$ only if the market is composed by $N$ linearly independent stocks.
\end{prop}
The condition in Proposition \ref{complete} is fulfilled by supposing that the vectors $\delta_{j,t}$'s representing the dividends are linearly independent.
\subsection{The Discounted Super-hedging Portfolio Value}
\indent Suppose $(V, h, K)$ is the unique superhedging strategy for the American option with payoff $G$. Let $\varphi_t$, $t \in [0,T]$, be the matrix whose $i$-th columns are $s_{i,t}$, $i=1,\cdots,N$. Then, from \eqref{dV}, $(V,h,K)$ satisfies:
\begin{equation}\label{dV2}
dV_t = r_tV_t dt +  X_t' (-r_t + (A'_t -\Gamma'_t))\varphi_t h_tdt + h'_{t-} \varphi'_{t-} dM_t -dK_t.
\end{equation}
 We shall write the equation for the discounted portfolio $\pi V$. Using the product rule for semimartingales, we have:
\begin{equation*}
V_t\pi_t = V_T\pi_T - \int_t^T V_{u_{-}} d\pi_u - \int_t^T \pi_{u_{-}}dV_u - \sum_{t < u \leq T} \Delta \pi_u \Delta V_u,
\end{equation*}
and we recall from Chapter 1 that $\sum\limits_{t < u \leq T} \Delta \pi_u \Delta V_u$ is the optional covariation of $\pi_t$ and $V_t$. Note again $dX_t = \Delta X_t$ and $\Delta X_t = \Delta M_t$.
We have from \eqref{pi} and \eqref{dV2} that $$\Delta \pi_t = \pi_{t_{-}} X'_{t_{-}} \sigma_{t_{-}} \Delta X_t ~~ \text{and}~~\Delta V_t = h'_{t-} \varphi'_{t-} \Delta X_t.$$
Also,
\begin{align*}
 &\pi_{u_{-}} X'_{u_{-}} \sigma_{u_{-}} \Delta X_u h'_{u-} \varphi'_{u-} \Delta X_u \\
& = \sum_{i,j}  \pi_{u_{-}}   (e'_iX_{u_{-}})  (e'_j X_u) (e'_i \sigma_{u_{-}} (e_j -e_i))h'_{u-} \varphi'_{u-}(e_j-e_i) \\
& = \sum_{i,j} \pi_{u_{-}} (e'_iX_{u_{-}}) (e'_j \Delta X_u) (e'_i \sigma_{u_{-}} (e_j -e_i))h'_{u-} \varphi'_{u-}(e_j-e_i).
\end{align*}
 Therefore, noting $\sigma_u^{ii}=0, ~ i=1,2,\cdots,N$,
\begin{align*}
&\sum_{t < u \leq T} \Delta \pi_u \Delta V_u \\
& = \sum_{i,j} \sum_{t < u \leq T} \pi_{u_{-}} (e'_iX_{u_{-}}) (e'_j \Delta X_u) (e'_i \sigma_{u_{-}} (e_j -e_i))h'_{u-} \varphi'_{u-}(e_j-e_i) \\
& = \int_t^T \sum_{i,j} \pi_{u_{-}} (e'_iX_{u_{-}}) ( e'_j (A_uX_u du + dM_u))(e'_i \sigma_{u_{-}} (e_j -e_i))h'_{u-} \varphi'_{u-}(e_j-e_i) \\
& = \int_t^T \sum_{i,j} \pi_{u_{-}} (e'_iX_{u}) e'_j(A_uX_u)(e'_i \sigma_{u} (e_j -e_i))h'_{u-} \varphi'_{u-}(e_j-e_i)  du\\
& + \int_t^T \sum_{i,j} \pi_{u_{-}} (e'_iX_{u}) (e'_jdM_u)(e'_i \sigma_{u_{-}} (e_j -e_i))h'_{u-} \varphi'_{u-}(e_j-e_i)\\
& = \int_t^T \sum_{i,j} \pi_{u} (e'_iX_{u}) A_u^{ji}\sigma_u^{ij} h'_{u} \varphi'_{u}(e_j-e_i)  du\\
& + \int_t^T \sum_{i,j} \pi_{u_{-}} (e'_iX_{u})(e'_jdM_u) \sigma_{u-}^{ij}h'_{u-} \varphi'_{u-}(e_j-e_i).
\end{align*}
Hence we derive
\begin{align*}
\pi_t V_t &= \pi_T V_T + \int_t^T V_u \pi_u r_u du - \int_t^T V_{u_{-}}\pi_{u_{-}}X'_{u_{-}}\sigma_{u_{-}}dM_u  \\
&  -  \int_t^T \pi_u V_u \ r_u du - \int_t^T \pi_u X'_u (-r_u + (A'_u - \Gamma'_u))\varphi_u h_u du \\
& - \int_t^T \pi_u (-dK_u) - \int_t^T \pi_{u_{-}} h'_{u_{-}}\varphi'_{u-} dM_u \\
& - \int_t^T \sum_{i,j} \pi_{u_{-}} (e'_iX_{u_{-}}) A_u^{ji}\sigma_u^{ij} h'_{u-} \varphi'_{u-}(e_j-e_i)  du\\
& - \int_t^T \sum_{i,j} \pi_{u_{-}} (e'_iX_{u_{-}})(e'_jdM_u) \sigma_{u-}^{ij}h'_{u-} \varphi'_{u-}(e_j-e_i).
\end{align*}\\
Collecting together the $du$ terms, and the $dM_u$ terms,
we have:
\begin{align*}
&\pi_t V_t \\
&= \pi_T V_T + \int_t^T (\pi_u X'_u (-r_u + (A'_u - \Gamma'_u))\varphi_u h_u - \pi_u\sum_{i,j} (X'_{u}e_i) A_u^{ji}\sigma_u^{ij} h'_u \varphi'_u (e_j-e_i)) du \\
& + \int_t^T \pi_u dK_u \\
& - \int_t^T ( \pi_{u_{-}} V_{u_{-}}X'_{u_{-}} \sigma_{u_{-}} + \sum_{i,j}  \pi_{u_{-}}(e'_iX_{u-}) \sigma_{u-}^{ij}h'_{u-}\varphi'_{u-}(e_j-e_i) e'_j  + \pi_{u-} h'_{u-} \varphi'_{u-}) dM_u.
\end{align*}
Now, let $\tilde{V}_t = \pi_t V_t$, $\tilde{Z}_t = \pi_t \varphi_t h_t$ and $\tilde{K}_t = \int_0^t \pi_u dK_u$.  Also, let
\begin{equation*}
H(t, z) =  X'_t (-r_t + (A'_t - \Gamma'_t))z - \sum_{i,j} (X'_{t}e_i) A_t^{ji}\sigma_t^{ij} z' (e_j-e_i)),\quad \text{and}
\end{equation*}
\begin{equation*}
I(t, z, v) =  v X'_{t_{-}} \sigma_{t_{-}} + \sum_{i,j} (e'_iX_{t-}) \sigma_{u-}^{ij}z(e_j-e_i) e'_j  +  z'.
\end{equation*}
Then, $(\tilde{V}_t, \tilde{Z}_t, \tilde{K}_t)$ solves the following equation with final condition $\pi_T G_T$:
\begin{equation}\label{RBSDEdiscount}
\begin{cases}
1) & \tilde{V}_t = \pi_TG_T + \int_t^T H(u,\tilde{Z}_u ) du  + \tilde{K}_T - \tilde{K}_t - \int_t^T I(u_{-}, \tilde{Z}_{u_{-}}, \tilde{V}_{u_{-}}) dM_u; \\
2) & \tilde{V}_t \geq \pi_t G_t.
\end{cases}
\end{equation}
Such a solution is called a super-hedging strategy for the discounted American claim.
\begin{prop}
Consider $\mathcal{S}_t^T$, the set of all stopping times $\{\tau\}$ with $t \leq \tau \leq T$. Then the solution $\tilde{V}$ of \eqref{RBSDEdiscount} is the solution to the optimal stopping time problem:
\begin{equation*}
\tilde{V}_t = ess \sup_{\tau \in \mathcal{S}_t^T}   E \left[ \int_t^{\tau} H(u, \tilde{Z}_u) du  + \pi_{\tau}G_{\tau}1_{\{\tau < T\}} +\pi_T G_T1_{\{\tau=T\}} | \mathcal{F}_t\right] .
\end{equation*}
\end{prop}
\begin{proof}
Let $\tau \in \mathcal{S}_t^T$. Take the conditional expectation from time $t$ to time $\tau$ in \eqref{RBSDEdiscount}:
\begin{align*}
\tilde{V}_t &= E \left[ \int_t^{\tau} H(u, \tilde{Z}_u) du + \tilde{V}_{\tau} + \tilde{K}_{\tau}   - \tilde{K}_t| \mathcal{F}_t \right] .
\end{align*}
Since $\tilde{K}_{\tau} - \tilde{K}_t \geq 0$ and $$ \tilde{V}_{\tau} \geq  \pi_{\tau}G_{\tau}1_{\{\tau < T\}} + \pi_T G_T1_{\{\tau=T\}},$$
\begin{equation*}
\tilde{V}_t \geq E \left[ \int_t^{\tau} H(u, \tilde{Z}_u) du  + \pi_{\tau}G_{\tau}1_{\{\tau < T\}} + \pi_T G_T1_{\{\tau=T\}}| \mathcal{F}_t\right].
\end{equation*}
This is true for any $\tau \in \mathcal{S}_t^T$, in particular:
\[\tilde{V}_t \geq ess \sup_{\tau \in \mathcal{S}_t^T}   E \left[ \int_t^{\tau} H(u, \tilde{Z}_u) du  + \pi_{\tau}G_{\tau}1_{\{\tau < T\}} +\pi_T G_T1_{\{\tau=T\}}| \mathcal{F}_t\right] .\]
The reverse of the above inequality is obtained by choosing an optimal time from $\mathcal{S}_t^T$ and the condition $\int_0^T (V_t-G_t) dK_t =0$ . In fact, let
\begin{equation*}
\tau_t = \inf \{t \leq u \leq T; V_u = G_u\},
\end{equation*}
and $\tau_t = T$ if $V_u \geq G_u$. When $t\leq s< \tau_t$, $V_t > G_t$, therefore $dK_u = 0 $ for $t\leq s< \tau_t$. Taking the integral from $t$ to $\tau_t$ and using the continuity of $K$,  we have
\[\tilde{K}_{\tau_t}-\tilde{K}_t =  \int_t^{\tau_t} \pi_u dK_u = 0.\]
Therefore:
\begin{align*}
\tilde{V}_t & = E\left[ \pi_{\tau_t} G_{\tau_t}+ \int_t^{\tau_t} H(u, \tilde{Z}_u) du + \tilde{K}_{\tau_t} - \tilde{K}_{t}| \mathcal{F}_t \right] \\
& = E\left[ \pi_{\tau_t}G_{\tau_t} + \int_t^{\tau_t} H(u, \tilde{Z}_u) du | \mathcal{F}_t \right] \\
& \leq ess \sup_{\tau \in \mathcal{S}_t^T} E\left[  \int_t^{\tau} H(u, \tilde{Z}_u) du  + \pi_{\tau}G_{\tau}1_{\{\tau < T\}} +\pi_T G_T1_{\{\tau=T\}}| \mathcal{F}_t \right].
\end{align*}
The price $\tilde{V}_t = \pi_t V_t$ is the super-replication of the discounted payoff $\pi_tG_t$ of the American option.
\end{proof}

\section{Conclusion}
\indent American options have been discussed in a market model where uncertainty is described by a Markov chain. RBSDEs are introduced in this framework and the existence and uniqueness of their solutions established. A constrained super-hedging strategy for an American option is shown to exist as the unique solution of an RBSDE.

\end{document}